\documentclass{amsart}
\usepackage{amsmath,amssymb,amsthm}
\usepackage[latin1]{inputenc}
\usepackage{tabularx,multicol}
\usepackage{graphicx,float,psfrag,array}
 \usepackage{stmaryrd,mathrsfs}
\usepackage{url}

\newcommand{\reff}[1]{(\ref{#1})}

\theoremstyle{plain}
\newtheorem{theo}{Theorem}[section]
\newtheorem{cor}[theo]{Corollary}
\newtheorem{prop}[theo]{Proposition}
\newtheorem{lem}[theo]{Lemma}
\newtheorem{defi}[theo]{Definition}
\newtheorem{theorem}[theo]{Theorem}
\newtheorem{definition}[theo]{Definition}

\newtheorem{lemma}[theo]{Lemma}

\newtheorem{proposition}[theo]{Proposition}
\theoremstyle{remark}
\newtheorem{rem}[theo]{Remark}

\renewcommand{\phi}{\varphi}
\renewcommand{\epsilon}{\varepsilon}

\newcommand{\ca}{{\mathcal A}}
\newcommand{\cb}{{\mathcal B}}
\newcommand{\cc}{{\mathcal C}}

\newcommand{\cm}{{\mathcal M}}

\newcommand{\crr}{{\mathcal R}}

\newcommand{\as}{{\mathfrak S}}

\newcommand{\ct}{{\mathcal T}}

\newcommand{\cx}{{\mathcal X}}
\newcommand{\cw}{{\mathcal W}}
\newcommand{\cz}{{\mathcal Z}}
\newcommand{\cy}{{\mathcal Y}}

\newcommand{\N}{{\mathbb N}}

\newcommand{\R}{{\mathbb R}}

\newcommand{\T}{{\mathbb T}}

\newcommand{\K}{{\mathbb K}}

\newcommand{\LL}{\mathbb L}

\newcommand{\ind}{{\bf 1}}

\newcommand{\diam}{{\rm diam}\;}

\newcommand{\val}[1]{\mathop{\left| #1 \right|}\nolimits}
\newcommand{\inv}[1]{\mathop{\frac{1}{ #1}}\nolimits}
\newcommand{\expp}[1]{\mathop {\mathrm{e}^{ #1}}}

\begin{document}
 
\title{A note on Gromov-Hausdorff-Prokhorov distance between (locally)
compact measure spaces} 
\date{\today}
\author{Romain Abraham} 

\address{
Romain Abraham,
MAPMO, CNRS UMR 7349,
F\'ed\'eration Denis Poisson FR 2964,
Universit\'e d'Orl\'eans,
B.P. 6759,
45067 Orl\'eans cedex 2
FRANCE.
}
 
\email{romain.abraham@univ-orleans.fr} 

\author{Jean-Fran\c{c}ois Delmas}
\address{
Jean-Fran\c cois Delmas,
Universit\'e Paris-Est, CERMICS, 6-8
av. Blaise Pascal,
 Champs-sur-Marne, 77455 Marne La Vall\'e, France.
\url{http://cermics.enpc.fr/~delmas/home.html}
}
\email{delmas@cermics.enpc.fr}

\author{Patrick Hoscheit}
\address{
Patrick Hoscheit, 
Universit\'e Paris-Est, CERMICS, 6-8
av. Blaise Pascal,
 Champs-sur-Marne, 77455 Marne La Vall\'e, France.
\url{http://cermics.enpc.fr/~hoscheip/home.html}
}
\email{hoscheip@cermics.enpc.fr}

\thanks{This work is partially supported the French ``Agence Nationale de
 la Recherche'', ANR-08-BLAN-0190.}



 \begin{abstract} 
 We present an extension of the Gromov-Hausdorff metric on the set of compact
metric spaces: the Gromov-Hausdorff-Prokhorov metric on the set of compact 
metric spaces endowed with a finite measure. We then extend it to the 
non-compact case by describing a metric on the set of rooted complete 
locally compact length spaces endowed with a locally finite measure. We 
prove that this space with the extended Gromov-Hausdorff-Prokhorov 
metric is a Polish space. This generalization is needed to define Lévy
trees, which are (possibly unbounded) random real trees endowed with a locally
finite measure.
\end{abstract}

\maketitle

\section{Introduction}

In the present work, we aim to give a topological framework to certain
classes of measured metric spaces. The methods go back to
ideas from Gromov~\cite{Gromov1999metric}, who first considered the
so-called Gromov-Hausdorff metric in order to compare metric spaces who
might not be subspaces of a common metric space. The
classical theory of the Gromov-Hausdorff metric on the space of compact
metric spaces, as well as its extension to locally compact spaces, is
exposed in particular in Burago, Burago and Ivanov~\cite{Burago2001}.

Recently, the concept of Gromov-Hausdorff convergence has found striking
applications in the field of probability theory, in the context of
random graphs. Evans \cite{Evans} and Evans, Pitman and Winter
\cite{Evans2005} considered the space of real trees, which is
Polish when endowed with the Gromov-Hausdorff metric. This has given a
framework to the theory of continuum random trees, which originated with
Aldous \cite{Aldous1991a}. There are also applications in the context
of random maps, where there have been significant developments in these
last years. In the monograph by Evans \cite{Evans}, the author
describes a topology on the space of compact real trees, equipped
with a probability measure, using the Prokhorov metric to compare
the measures, thus defining the so-called weighted Gromov-Hausdorff
metric. Recently Greven, Pfaffelhuber and Winter
\cite{Greven2008b} take another approach by considering the
space of complete, separable metric spaces, endowed with probability
measures (metric measure spaces). In order to compare two such
probability spaces, they consider embeddings of both these spaces into
some common Polish metric space, and use the Prokhorov metric to compare
the ensuing measures. This puts the emphasis on the probability measure
carried by the space rather than its geometrical features. In his
monograph, Villani \cite{Villani2009} gives an account of the theory
of measured metric spaces and the different approaches to their
topology. Miermont, in \cite{Miermont2007}, describes a combined
approach, using both the Hausdorff metric and the Prokhorov metric to
compare compact metric spaces equipped with probability measures. The
metric he uses (called the Gromov-Hausdorff-Prokhorov metric) is not the
same as Evans's, but they are shown to give rise to the same topology.

In the present paper, we describe several properties of the
Gromov-Hausdorff-Prokhorov metric, $d^c_{\text{GHP}}$, on the set $\K$
of (isometry classes of) compact metric spaces, with a distinguished element
called the root and endowed with a finite measure. Theorem
\ref{theo:dcGHP} ensures that $(\K, d^c_{\text{GHP}})$ is a Polish
metric space. We extend those results by considering the
Gromov-Hausdorff-Prokhorov metric, $d_{\text{GHP}}$, on the set $\LL$ of
(isometry classes of) rooted locally compact, complete length spaces, endowed
with
a locally finite measure. Theorem \ref{theo:LL} ensures that $(\LL,
d_{\text{GHP}})$ is also a Polish metric space. The proof of the
completeness of $\LL$ relies on a pre-compactness criterion given in
Theorem \ref{GHP:PreComL}. The methods used are similar to the methods
used in \cite{Burago2001} to derive properties about the
Gromov-Hausdorff topology of the set of locally compact complete length
spaces. This work extends some of the results from \cite{Greven2008b},
which doesn't take into account the geometrical structure of the spaces,
as well as the results from \cite{Miermont2007}, which consider only the
compact case and probability measures. This comes at the price of
having to restrict ourselves to the context of length spaces. In
\cite{Villani2009} the Gromov-Hausdorff-Prokhorov topology is considered
for general Polish spaces (instead of length spaces) but endowed with
locally finite measures satisfying the doubling condition. We also mention the
different approach of \cite{Addario-Berry2011}, using the ideas of
correspondences between metric spaces and couplings of measures.


This work was developed for applications in the setting of weighted
real trees (which are elements of $\LL$), see Abraham, Delmas and
Hoscheit \cite{Abraham2011a}. We give an hint of those applications by
stating that the construction of a weighted tree coded in a continuous
function with compact support is measurable with respect to the topology
induced by $d^c_{\text{GHP}}$ on $\K$ or by $d_{\text{GHP}}$ on
$\LL$. This construction allows us to define random variables on $\K$
using continuous random processes on $\R$, in particular the L\'evy
trees of \cite{Duquesne2005a} that describe the genealogy of the
so-called critical or sub-critical continuous state branching
processes that become a.s. extinct. The measure $\mathbf{m}$ is then
a ``uniform'' measure on the leaves of the tree which has finite mass. The
construction can be generalized to super-critical continuous state branching
processes which can live forever; in that case the corresponding
genealogical tree is infinite and the measure $\mathbf{m}$ on the
leaves is also infinite. This paper gives an appropriate framework to
handle such tree-valued random variables and also tree-valued Markov
processes as in \cite{Abraham2011a}.

The structure of the paper is as follow. Section \ref{sec:main} collects
the main results of the paper. The application to real trees is given in
Section \ref{sec:code}. The proofs of the results in the compact case
are given in Section \ref{sec:K}. The proofs of the results in the
locally compact case are given in Section \ref{sec:L}.

\section{Main results}
\label{sec:main}
\subsection{Rooted weighted metric spaces}

Let $(X,d^X)$ be a Polish metric space. The diameter of $A\in \cb(X)$ is
given by:
\[
\diam(A)=\sup\{d^X(x,y); \ x,y\in A\}.
\]
For $A,B\in
\cb(X)$, we set:
\[
 d_\text{H}^X(A,B)= \inf \{ \epsilon >0;\ A\subset B^\epsilon\
 \mathrm{and}\ B\subset 
A^\epsilon \}, 
\]
the Hausdorff metric between $A$ and $B$, where 
\begin{equation}
 \label{eq:e-halo}
A^\epsilon = \{ x\in
X;\ \inf_{y\in A} d^X(x,y) < \epsilon\}
\end{equation}
is the $\epsilon$-halo set of $A$. If $X$ is compact, then the space of
compact subsets of $X$, endowed with the Hausdorff metric, is compact,
see theorem 7.3.8 in \cite{Burago2001}. To give pre-compactness
criterion, we shall need the notion of $\epsilon$-nets.
\begin{definition}
\label{defi:e-net}
 Let $(X,d^X)$ be a metric space, and let $\epsilon >0$. A subset $A\subset X$
is an $\epsilon$-net of $B\subset X$ if:
\[
A\subset B\subset A^\varepsilon.
\] 
\end{definition}
Notice that, for any $\epsilon>0$, compact metric spaces admit finite
$\epsilon$-nets and locally compact spaces admit locally finite
$\epsilon$-nets.\\

Let $\cm_{f}(X)$ denote 
the set of all finite Borel measures on $X$. If $\mu,\nu \in
\cm_f(X)$, we set:
\[
 d_\text{P}^X(\mu,\nu) = \inf \{ \epsilon >0;\ \mu(A)\le \nu(A^\epsilon) +
 \epsilon 
\text{ and } 
\nu(A)\le \mu(A^\epsilon)+\epsilon\ 
\text{ for any closed set } A \}, 
\]
the Prokhorov metric between $\mu$ and $\nu$. It is well known,
see \cite{Daley2003} Appendix A.2.5, that
$(\cm_f(X), d_\text{P}^X)$ is a Polish metric space, and that the
topology generated by $d_\text{P}^X$ is exactly the topology of weak
convergence (convergence against continuous bounded functionals).

The Prokhorov metric can be extended in the following way. Recall that a Borel 
measure is locally finite if the measure of any
bounded Borel set is finite. Let $\cm(X)$
denote the set of all locally finite Borel measures on $X$. Let
$\emptyset$ be a distinguished element of $X$, which we shall call the
root. We will consider
 the closed ball of radius $r$ centered at
$\emptyset$: 
\begin{equation}
 \label{eq:X(r)}
X^{(r)}=\{x\in X; d^X(\emptyset,x)\leq r\},
\end{equation}
and for $\mu\in \cm(X)$ its restriction $\mu^{(r)}$ to $X^{(r)}$:
\begin{equation}
 \label{eq:mu(r)}
\mu^{(r)} (dx)=\ind_{X^{(r)}} (x)\; \mu(dx).
\end{equation}
If $\mu,\nu \in
\cm(X)$, we define a generalized Prokhorov metric between $\mu$ and $\nu$:
\begin{equation}
 \label{eq:dgP}
 d_\text{gP}^X(\mu,\nu) = \int_0^\infty \expp{-r} \left(1 \wedge
d^X_{\text{P}}\left(\mu^{(r)},\nu^{(r)}\right)
\right) \ dr.
\end{equation}
It is not difficult to check that $d_\text{gP}^X$ is well defined (see
Lemma \ref{lem:reg-d-GHP} in a more general framework) and is a
metric. Furthermore $(\cm(X), d_\text{gP}^X)$ is a Polish metric
space, and the topology generated by $d_\text{gP}^X$ is exactly the
topology of vague convergence (convergence against continuous bounded
functionals with
bounded support), see \cite{Daley2003} Appendix A.2.6.\\

When there is no ambiguity on the metric space $(X, d^X)$, we may write $d$,
$d_\text{H}$, and $d_\text{P}$ instead of $d^X$, $d^X_\text{H}$ and
$d^X_\text{P}$. In the case where we consider different metrics on the same
space, in order to stress that the metric is $d^X$, we shall write
$d^{d^X}_\text{H}$ and $d^{d^X}_\text{P}$ for the
corresponding Hausdorff and Prokhorov metrics.

If $\Phi:X\rightarrow X'$ is a Borel map between two Polish metric
spaces and if $\mu$ is a Borel measure on $X$, we will note $\Phi_*\mu$
the image measure on $X'$ defined by $\Phi_*\mu(A)=\mu(\Phi^{-1}(A))$,
for any Borel set $A\subset X$. 

\begin{defi}
 \label{defi:rwms} 
$ $
\begin{itemize}
\item A rooted weighted metric space $\cx = (X,d, \emptyset,\mu)$ is a
 metric space $(X , d)$ with a distinguished element $\emptyset\in X$,
 called the root, and a locally finite Borel measure $\mu$.
\item Two rooted weighted metric spaces $\cx=(X,d,\emptyset,\mu)$
 and $\cx'=(X',d',\emptyset',\mu') $ are said to be
 GHP-isometric if there exists an isometric one-to-one map $\Phi:X
 \rightarrow X'$ such that $\Phi(\emptyset)= \emptyset'$ and 
 $\Phi_* \mu = \mu'$. In that case, $\Phi$ is called a GHP-isometry. 
\end{itemize}
\end{defi}

Notice that if $(X, d)$ is compact, then a locally finite measure on
$X$ is finite and belongs to $\cm_f(X)$. We will now use a 
procedure due to Gromov \cite{Gromov1999metric} to compare any two
compact rooted weighted metric spaces, even if they are not subspaces of
the same Polish metric space.

\subsection{Gromov-Hausdorff-Prokhorov metric for compact spaces}
For convenience, we recall the Gromov-Hausdorff metric, see for
example Definition 7.3.10 in \cite{Burago2001}. Let $(X,d)$ and
$(X',d')$ be two compact metric spaces. The Gromov-Hausdorff metric
between $(X,d)$ and $(X',d')$ is given by:
\begin{equation}
 \label{eq:d-GH}
 d_{\text{GH}}^c((X,d), (X',d')) = \inf_{\varphi,\varphi',Z} 
 d_\text{H}^Z(\varphi(X),\varphi'(X')) , 
\end{equation}
where the infimum is taken over all isometric embeddings
$\varphi:X\hookrightarrow Z$ and $\varphi':X'\hookrightarrow Z$ into some
common Polish metric space $(Z,d^Z)$. 
Note that Equation \reff{eq:d-GH} does
actually define a metric on the set of isometry classes of
compact metric spaces.\\

Now, we introduce the Gromov-Hausdorff-Prokhorov metric for compact
spaces. Let $\cx=(X,d,\emptyset,\mu)$ and
$\cx'=(X',d',\emptyset',\mu')$ be two compact rooted weighted metric
spaces, and define:
\begin{equation} 
\label{f:def}
 d_{\text{GHP}}^c(\cx,\cx') = \inf_{\Phi,\Phi',Z} \left(
 d^Z(\Phi(\emptyset),\Phi'(\emptyset')) + d_\text{H}^Z(\Phi(X),\Phi'(X')) + 
d_\text{P}^Z(\Phi_* \mu,\Phi_*'
\mu') \right), 
\end{equation}
where the infimum is taken over all isometric embeddings
$\Phi:X\hookrightarrow Z$ and $\Phi':X'\hookrightarrow Z$ into some
common Polish metric space $(Z,d^Z)$. 

Note that equation \reff{f:def} does
not actually define a metric, as
$d_{\text{GHP}}^c(\cx,\cx')=0$ if $\cx$ and $\cx'$ are GHP-isometric.
Therefore, we shall consider $\K$, the set of GHP-isometry classes of
compact rooted weighted metric space and identify a compact rooted
weighted metric space with its class in $\K$. Then the function
$d_{\text{GHP}}^c$ is finite on $\K^2$. 

\begin{theo}
 \label{theo:dcGHP}
$ $ 
\begin{itemize}
 \item[(i)] The function $d_{\text{GHP}}^c$ defines a metric on $\K$.
 \item[(ii)] The space
$(\K, d_{\text{GHP}}^c)$ is a Polish metric space.
\end{itemize} 
 \end{theo}

We shall call $d_{\text{GHP}}^c$ the Gromov-Hausdorff-Prokhorov
metric. This extends the Gromov-Hausdorff metric on compact metric
spaces, see \cite{Burago2001} section 7, as well as the
Gromov-Hausdorff-Prokhorov metric on compact metric spaces endowed with a
probability measure, see \cite{Miermont2007}. See also
\cite{Greven2008b} for another approach on metric spaces endowed with a
probability measure. 

We end this Section by a pre-compactness criterion on $\K$. 

\begin{theorem}
\label{GHP:PreComK}
 Let $\ca$ be a subset of $\K$, such that:
\begin{itemize}
\item[(i)] We have $\sup_{(X,d,\emptyset,\mu) \in \ca} \diam(X )<+\infty
 $.
\item[(ii)] For every $\varepsilon>0$, there exists a finite integer
 $N(\epsilon)\ge 1$, such that for any $(X,d,\emptyset,\mu) \in \ca$,
 there is an $\varepsilon$-net of $X$ with cardinal less than
 $N(\varepsilon)$.
\item[(iii)] We have $\sup_{(X,d,\emptyset,\mu) \in \ca} \mu(X )<+\infty
 $.
\end{itemize}
Then, $\ca$ is relatively compact: every sequence in $\ca$ admits a
sub-sequence that converges in the $d^c_{\text{GHP}}$ topology.
\end{theorem}

Notice that we could have defined a Gromov-Hausdorff-Prokhorov metric
without reference to any root. However, the introduction of the root is
necessary to define the Gromov-Hausdorff-Prokhorov metric for locally
compact spaces, see next Section.

\subsection{Gromov-Hausdorff-Prokhorov metric for locally compact spaces}

To consider an extension to non compact weighted rooted metric spaces,
we shall consider complete and locally compact length spaces.

We recall that a
metric space $(X,d)$ is a length space if for every $x,y\in X$,
we have:
\[
 d(x,y) = \inf L(\gamma) , 
\]
where the infimum is taken over all rectifiable curves $\gamma:[0,1]\rightarrow
X$ such that $\gamma(0)=x$ and $\gamma(1)=y$, and where $L(\gamma)$ is the
length of the rectifiable curve $\gamma$. We recall that $(X,d)$ is a
length space if is satisfies the mid-point condition (see Theorem 2.4.16 in
\cite{Burago2001}): for all
$\varepsilon>0$, $x,y\in X$, there exists $z\in X$ such that:
\[
\val{2d(x,z) - d(x,y)}+\val{2d(y,z) - d(x,y)}\leq \varepsilon.
\]

\begin{defi}
 \label{defi:L}
Let $\LL$ be the set of GHP-isometry classes of rooted, weighted, complete
and locally compact length spaces and identify a rooted, weighted, complete
and locally compact length spaces with its class in $\LL$.
\end{defi}

If $\cx=(X,d,\emptyset,\mu)\in \LL$, then for $r\geq 0$ we will consider
its restriction to the closed ball of radius $r$ centered at
$\emptyset$, $\cx^{(r)}=(X^{(r)}, d^{(r)}, \emptyset, \mu^{(r)})$, where
$X^{(r)} $ is defined by \reff{eq:X(r)}, the metric $d^{(r)}$ is the
restriction of $d$ to $X^{(r)}$, and the measure $\mu^{(r)}$ is defined
by \reff{eq:mu(r)}. Recall that the Hopf-Rinow theorem implies that if $(X,
d)$ is a complete and locally compact length space, then every closed
bounded subset of $X$ is compact. In particular if $\cx$ belongs to $
\LL$ , then $\cx^{(r)}$ belongs to $\K$ for all $r\geq 0$.

We state a regularity Lemma of $d^c_{\text{GHP}}$ with respect to the
restriction operation. 

\begin{lem}
 \label{lem:reg-d-GHP}
Let $\cx$ and $\cy$ be in $\LL$. Then the function defined on $\R_+$
by $ r\mapsto d^c_{\text{GHP}}\left(\cx^{(r)},\cy^{(r)}\right)$
is càdlàg.
\end{lem}
This implies that the following function (inspired by \reff{eq:dgP}) is
well defined on $\LL^2$:
\[
 d_{\text{GHP}}(\cx,\cy) = \int_0^\infty \expp{-r} \left(1 \wedge
d^c_{\text{GHP}}\left(\cx^{(r)},\cy^{(r)}\right)
\right) \ dr.
\]


\begin{theo}
 \label{theo:LL}
$ $
\begin{itemize}
 \item[(i)] The function $d_{\text{GHP}}$ defines a metric on $\LL$.
\item[(ii)] The space $(\LL, d_{\text{GHP}})$ is a Polish metric space.
\end{itemize}
\end{theo}

The next result implies that $d_{\text{GHP}}^c$ and $d_{\text{GHP}}$
define the same topology on $\K\cap \LL$. 
\begin{prop}
 \label{prop:K-LL} Let $(\cx_n, n\in \N)$ and $\cx$ be elements of $\K\cap
 \LL$. Then the sequence $(\cx_n, n\in \N)$ converges to $\cx$ in
 $(\K,d_{\text{GHP}}^c)$ if and only if it converges to $\cx$ in
 $(\LL,d_{\text{GHP}})$. 
\end{prop}

Finally, we give a pre-compactness criterion on $\LL$ which is a
generalization of the well-known compactness theorem for compact metric
spaces, see for instance Theorem 7.4.15 in \cite{Burago2001}.

\begin{theorem}
\label{GHP:PreComL}
 Let $\cc$ be a subset of $\LL$, such that for every $r\ge 0$: 
\begin{itemize}
\item[(i)] For every $\varepsilon>0$, there exists a finite integer
 $N(r,\epsilon)\ge 1$, such that for any $(X,d,\emptyset,\mu) \in \cc$,
 there is an $\varepsilon$-net of $X^{(r)}$ with cardinal less than
 $N(r,\varepsilon)$. 
 \item[(ii)] We have $\sup_{(X,d,\emptyset,\mu) \in \cc} \mu(X^{(r)} )<+\infty
 $. 
\end{itemize}
Then, $\cc$ is relatively compact: every sequence in $\cc$ admits a
sub-sequence that converges in the $d_{\text{GHP}}$ topology.
\end{theorem}

\section{Application to real trees coded by functions}
\label{sec:code}

A metric space $(T,d)$ is a called real tree (or
$\R$-tree) if the following properties are satisfied:
\begin{itemize}
	\item[(i)] For every $s,t\in T$, there is a unique isometric map
$f_{s,t}$
from $[0,d(s,t)]$ to $T$ such that $f_{s,t}(0)=s$ and $f_{s,t}(d(s,t))=t$. 
	\item[(ii)] For every $s,t\in T$, if $q$ is a continuous injective map
from
$[0,1]$ to $T$ such that $q(0)=s$ and $q(1)=t$, then
$q([0,1])=f_{s,t}([0,d(s,t)])$. 
\end{itemize}
Note that real trees are always length spaces and that complete real
trees are the only complete connected spaces that satisfy the so-called
four-point condition:
\begin{equation}
 \forall x_1,x_2,x_3,x_4 \in X,\ d(x_1,x_2)+d(x_3,x_4) \le
(d(x_1,x_3)+d(x_2,x_4) ) \vee (d(x_1,x_4)+d(x_2,x_3)). 
\end{equation}

We say that a real tree is rooted if there is a distinguished vertex
$\emptyset$, which will be called the root of $T$. 

\begin{definition} 
 We denote by $\T$ the set of (GHP-isometry classes of) rooted, weighted,
complete and locally compact real trees, in short w-trees.
\end{definition}

We deduce the following Corollary from Theorem \ref{theo:LL} and the
four-point condition characterization of real trees. 

\begin{cor}
 \label{cor:T}
The set $\T$ is a closed subset of $\LL$ and 
$(\T,d_{\text{GHP}})$ is a Polish metric space. 
\end{cor}

Let $f$ be a continuous non-negative function defined on $[0,+\infty)$,
such that $f(0)=0$, with compact support. We set:
\[
\sigma^f=\sup\{t; f(t)>0\},
\]
with the convention $\sup\emptyset=0$. Let $d^f$ be the non-negative
function defined by:
\[ 
d^f(s,t) = f(s) + f(t) - 2 \inf_{u\in [s\wedge t , s\vee t]} f(u). 
\]
It can be easily checked that $d^f$ is a semi-metric on $[0,\sigma^f]$. One can
define the equivalence relation associated with $d^f$ by $s\sim t$ if and only
if
$d^f(s,t)=0$. Moreover, when we consider the quotient space
\begin{equation*}
 T^f=[0,\sigma^f]_{/ \sim}
\end{equation*}
and, noting again $d^f$ the induced metric on $T^f$ and rooting $T^f$ at
$\emptyset^f$, the equivalence class of 0, it can be checked that the
space $(T^f,d^f,\emptyset^f)$ is a rooted compact real tree. We denote
by $p^f$ the canonical projection from $[0,\sigma^f]$ onto $T^f$, which
is extended by $p^f(t)=\emptyset^f$ for $t\geq \sigma^f$. Notice that
$p^f$ is continuous. We define $\mathbf{m}^{f}$, the Borel measure on
$T^f$ as the image measure on $T^f$ of the Lebesgue measure on $[0,
\sigma^f]$ by $p^f$. We consider the (compact) w-tree $\ct^f=(T^f, d^f,
\emptyset^f, \mathbf{m}^f)$.

We have the following elementary result (see Lemma 2.3 of
\cite{Duquesne2005a} when dealing with the Gromov-Hausdorff metric instead of
the Gromov-Hausdorff-Prokhorov metric). For a proof, see \cite{Abraham2011a}.
\begin{prop}
 \label{prop:cont-TH}
 Let $f,g$ be two compactly supported, non-negative continuous functions with
$f(0)=g(0)=0$. Then, we have:
\begin{equation}
 d_{\text{GHP}}^c(\ct^f,\ct^g) \le 6 \| f-g \|_\infty + | \sigma^f-\sigma^g |.
\end{equation}
\end{prop}

This result and Proposition \ref{prop:K-LL} ensure that the map
$f\mapsto \ct^f$ (defined on the space of continuous functions with compact
support which vanish at 0, with the uniform topology) taking
values in $(\T\cap \K, d^c_{\text{GHP}})$ or $(\T, d_{\text{GHP}})$ is
measurable.

\section{Gromov-Hausdorff-Prokhorov metric for compact metric
 spaces} 
\label{sec:K}

\subsection{Proof of (i) of Theorem \ref{theo:dcGHP}}
In this Section, we shall prove 
that $d_{\text{GHP}}^c$ defines a metric on $\K$. 

First, we will prove the following technical lemma, which is a
generalization of Remark 7.3.12 in \cite{Burago2001}. Let $\cx=(X,d^X,
\emptyset^X,\mu^X)$ and $\cy=(Y, d^Y,\emptyset^Y, \mu^Y)$ be two
elements of $\K$. We will use the notation $X\sqcup Y$ for the disjoint
union of the sets $X$ and $Y$. We will abuse notations and note
$X,\mu^X,\emptyset^X$ and $Y,\mu^Y,\emptyset^Y$ the images of
$X,\mu^X,\emptyset^X$ and of $Y,\mu^Y,\emptyset^Y$ respectively by
the canonical embeddings $X \hookrightarrow X\sqcup Y$ and $Y\hookrightarrow
X\sqcup Y$.

\begin{lemma} 
\label{df:uniondisjointe} 
Let $\cx=(X,d^X,
 \emptyset^X,\mu^X)$ and $\cy=(Y, d^Y,\emptyset^Y, \mu^Y)$ be two
 elements of $\K$. Then, we have:
\begin{equation} 
 d_{\text{GHP}}^c(\cx,\cy) = \inf_{d} \left\{
d(\emptyset^X,\emptyset^Y) + 
d_\text{H}^d(X,Y) + d_\text{P}^d ( \mu^X,\mu^Y) \right\},
\end{equation}
where the infimum is taken over all metrics $d$ on $X\sqcup Y$ such that the
canonical embeddings $X \hookrightarrow X\sqcup Y$ and $Y\hookrightarrow X\sqcup
Y$ are isometries. 
\end{lemma}
\begin{proof}
We only have to show that: 
\begin{equation} \label{eq:inf}
\inf_{d} \left\{
d(\emptyset^X,\emptyset^Y) + 
d_\text{H}^d(X,Y) + d_\text{P}^d ( \mu^X,\mu^Y) \right\} \le
d_{\text{GHP}}^c(\cx,\cy), 
\end{equation}
since the other inequality is obvious. Let $(Z,d^Z)$ be a Polish
space and
$\Phi^X$ and $\Phi^Y$ be isometric embeddings of $X$ and $Y$ in $Z$.
Let $\delta>0$. We define the following function on $(X\sqcup Y)^2$:
\begin{equation} d(x,y) = 
\begin{cases}
 d^Z(\Phi^X(x),\Phi^Y(y)) + \delta & \text{if}\ x\in X,\ y\in Y, \\
 d^X(x,y) & \text{if}\ x,y\in X, \\
 d^Y(x,y) & \text{if}\ x,y\in Y.
\end{cases}
\end{equation}
It is obvious that $d$ is a metric on $X\sqcup Y$, and that the
canonical embeddings of $X$ and $Y$ in $X\sqcup Y $ are isometric.
Furthermore, by definition, we have $d(\emptyset^X,\emptyset^Y)=
d^Z(\Phi^X(\emptyset^X) ,\Phi^Y(\emptyset^Y)) +\delta$. Concerning the
Hausdorff distance between $X$ and $Y$, we get that:
\[
d_\text{H}^d(X,Y)\leq d_\text{H}^Z ( \Phi^X( X), \Phi^Y( Y))+\delta.
\]

Finally, let us compute the Prokhorov distance between
$\mu^X$ and $\mu^Y$. Let $\epsilon>0$ be such that $d_\text{P}^{Z}
(\Phi^X_*\mu^X,\Phi^Y_*\mu^Y) < \epsilon$. Let $A$ be a closed
subset of $X\sqcup Y$. By definition, we have:
\begin{align*}
 \mu^X(A) = \mu^X(A\cap X) 
& = \Phi^X_*\mu^X(\Phi^X(A\cap X)) \\
 & < \Phi^Y_* \mu^Y(\{z\in Z,\ d^Z(z,\Phi^X(A\cap X)) < \epsilon \}) +
 \epsilon\\ 
 & = \Phi^Y_*\mu^Y(\{z\in \Phi^Y( Y),\ d^Z(z,\Phi^X(A\cap X)) < \epsilon \}) +
\epsilon
\\
 & \le \mu^Y (\{ y\in Y,\ d(y,A \cap X) < \epsilon
+\delta \}) +\epsilon\\
& \leq \mu^Y (\{ y\in X\sqcup Y,\ d(y,A) < \epsilon
+\delta \}) +\epsilon. 
\end{align*}
The symmetric result holds for $(X,Y)$ replaced by $(Y,X)$ and therefore
we get $ d_\text{P}^d(\mu^X,\mu^Y) < \epsilon + \delta$. This implies:
\[
d_\text{P}^d(\mu^X,\mu^Y)\leq d_\text{H}^Z ( \Phi^X_*\mu^ X,
\Phi^Y_*\mu^Y)+\delta.
\]
Eventually, we get:
\begin{multline*}
d(\emptyset^X,\emptyset^Y)+ d_\text{H}^d(X,Y)+d_\text{P}^d(\mu^X,\mu^Y)\\
\leq d^Z(\Phi^X(\emptyset^X) ,\Phi^Y(\emptyset^Y))+ d_\text{H}^Z ( \Phi^X( X),
\Phi^Y( Y))+ d_\text{H}^Z ( \Phi^X_*\mu^ X, \Phi^Y_*\mu^Y)+3\delta.
\end{multline*}
Thanks to \reff{f:def} and since $\delta>0$ is arbitrary, we get \reff{eq:inf}.
\end{proof}

We now prove that $d_{\text{GHP}}^c$ does indeed satisfy all the axioms of a
metric (as is done in \cite{Burago2001} for the Gromov-Hausdorff metric
and in \cite{Miermont2007} in the case of probability measures on
compact metric spaces). The symmetry and positiveness of $d_{\text{GHP}}^c$
being obvious, let us prove the triangular inequality and positive
definiteness.

\begin{lemma}
 The function $d_{\text{GHP}}^c$ satisfies the triangular identity on $\K$. 
\end{lemma}
\begin{proof}
 Let $\cx_1,\cx_2$ and $\cx_3$ be elements of $\K$. For $i\in \{1,3\}$,
 let us assume that $d_{\text{GHP}}^c(\cx_i,\cx_2)<r_i$. With obvious
 notations, for $i\in \{1,3\}$, we consider, as in Lemma
 \ref{df:uniondisjointe}, metrics $d_i$ on $X_i\sqcup X_2$. Let us then
consider $Z=X_1\sqcup X_2\sqcup X_3$, on which we define:
\begin{equation}
 d(x,y) = 
\begin{cases}
 d_{i}(x,y) & \text{if}\ x,y\in (X_i\sqcup X_2)^2 \text{ for } i\in \{1,3\}, \\
 \inf_{z\in X_2} \{ d_{1}(x,z) + d_{3}(z,y) \} & \text{if}\ x\in X_1,y\in
X_3.
\end{cases}
\end{equation}
The function $d$ is in fact a metric on $Z$, and
the canonical embeddings are isometries, since they are for $d_{1}$ and
$d_{3}$. By definition, we have:
\[
 d_\text{H}^{d}(X_1,X_3) = \left( \sup_{x_1\in X_1} \inf_{x_3\in X_3}
d(x_1,x_3) \right) \vee \left( \sup_{x_3\in X_3} \inf_{x_1\in X_1}
d(x_1,x_3) \right) .
\]
We notice that:
\begin{align*}
\sup_{x_1\in X_1} \inf_{x_3\in X_3}
d(x_1,x_3) 
&= \sup_{x_1\in X_1} \inf_{x_2\in X_2,\ x_3\in X_3} d_{1}(x_1,x_2)
+ d_{3}(x_2,x_3)\\
 & \le d_\text{H}^{d_1 } (X_1,X_2) + \inf_{x_2\in X_2,\ x_3\in X_3}
d_{3}(x_2,x_3) \\
 & \le d_\text{H}^{d_1}(X_1,X_2) + d_\text{H}^{d_3}(X_2,X_3). 
\end{align*}
Thus, we deduce that $d_\text{H}^{d}(X_1,X_3) \leq d_\text{H}^{d_1}(X_1,X_2) +
d_\text{H}^{d_3}(X_2,X_3)$. 

As far as the Prokhorov distance is concerned, for $i\in \{1,3\}$, let
$\epsilon_i$ be such that $d_\text{P}^{d_i} (\mu_i,
\mu_2)<\varepsilon_i$. Then, if $A\subset Z$ is closed, we have:
\begin{align*}
 \mu_1(A) = \mu_1(A\cap X_1) 
& < \mu_2(\{ x\in X_1\sqcup X_2,\ d_1(x,A\cap X_1) <\epsilon_1 \}) +\epsilon_1
\\ 
 & \le \mu_2(A^{\epsilon_1} \cap X_2) + \epsilon_1 \\
 & < \mu_3(\{x\in X_3\sqcup X_2,\ d_3(x,A^{\epsilon_1} \cap
 X_2)<\varepsilon_3 \}) + \epsilon_1+\epsilon_3\\
 &\leq \mu_3(A^{\epsilon_1+\epsilon_3} ) +\epsilon_1+\epsilon_3,
\end{align*}
where $A^{\varepsilon}=\{z\in Z,\ d(z,A)<\varepsilon\}$, for
$\varepsilon=\varepsilon_1$ and $\varepsilon=\varepsilon_1+\varepsilon_3$. 
A similar result holds with $(\mu_1, \mu_3)$ replaced by $(\mu_3, \mu_1)$. We
deduce that $d_\text{P}^{d}(\mu_1,\mu_3) < \epsilon_1+\epsilon_3$, which
implies that $d_\text{P}^{d}(\mu_1,\mu_3) \le d_\text{P}^{d_{1}}(\mu_1,\mu_2)
+d_\text{P}^{d_{3}}(\mu_2,\mu_3)$. 

By summing up all the results, we get:
\[
d(\emptyset_1, \emptyset_3)+ d_\text{H}^d(X_1, X_3)+d_\text{P}^d(\mu_1, \mu_3)
\leq 
\sum_{i\in \{1, 3\}} d^{d_i} (\emptyset_i, \emptyset_2)+ d_\text{H}^{d_i} (X_i,
X_2)+d_\text{P}^{d_i}(\mu_i, \mu_2) .
\]
Then use the definition \reff{f:def} and Lemma \ref{df:uniondisjointe} to
get the triangular inequality:
\[
d_{\text{GHP}}^c(\cx_1, \cx_3)\leq d_{\text{GHP}}^c(\cx_1,
\cx_2)+d_{\text{GHP}}^c(\cx_2,
\cx_3).
\]
\end{proof}

This proves that $d_{\text{GHP}}^c$ is a semi-metric on $\K$. 
We then prove the positive
definiteness.

\begin{lemma}
\label{lem:dc=0}
 Let $\cx,\cy$ be two elements of $\K$ such that
 $d_{\text{GHP}}^c(\cx,\cy)=0$. Then $\cx=\cy$ (as GHP-isometry classes of
 rooted weighted compact metric spaces). 
\end{lemma}
\begin{proof}
 Let $\cx=(X, d^X,\emptyset^X, \mu^X)$ and $\cy=(Y, d^Y,\emptyset^Y,
 \mu^Y)$ in $\K $ such that $d_{\text{GHP}}^c(\cx,\cy)=0$. According to Lemma
 \ref{df:uniondisjointe}, we
 can find a sequence of metrics $(d^n,n\ge 1)$ on $X\sqcup Y$, such
 that
\begin{equation}
 \label{dfpetit} 
d^n(\emptyset^X,\emptyset^Y) + d_\text{H}^n(X,Y) + d_\text{P}^n(\mu^X,\mu^Y) <
\epsilon_n,
\end{equation}
for some positive sequence $(\epsilon_n,n\ge 1)$ decreasing to 0, where
$d^n_\text{H}$ and $d^n_\text{P}$ stand for $d^{d^n}_\text{H}$ and
$d^{d^n}_\text{P}$ . For any $k\ge
1$, let $S_k$ be a finite $(1/k)$-net of $X$, containing the root. Since
$X$ is compact, we get by Definition \ref{defi:e-net}
that $S_k$ is in fact an $(\inv{k}-\delta)$-net of $X$ for
some $\delta>0$. 
Let
$N_k+1$ be the cardinal of $S_k$. We will write:
\[ 
S_k = \{ x_{0,k}=\emptyset^X,x_{1,k},...,x_{N_k,k} \}. 
\]
Let $(V_{i,k}, 0\le i\le N_k)$ be Borel subsets of $X$ with diameter
less than $1/k$, that is: 
\[
\sup_{x,x'\in V_{i,k}} d^X(x,x')<1/k,
\]
such that $\bigcup _{0\le i\le N_k} V_{i,k}=X$ and 
for all $0\le i,i'\le N_k$, we have $V_{i,k}\bigcap
V_{i',k}=\emptyset$ and $x_{i,k}\in
V_{i,k}$ if $V_{i,k}\neq \emptyset$. We set: 
\[
\mu^X_k(dx)=\sum_{i=0}^{N_k} \mu^X(V_{i,k}) \delta_{x_{i,k}}(dx),
\]
where $\delta_{x'}(dx)$ is the Dirac measure at $x'$. Notice that:
\[
d_\text{H}^X(X,S_k) \le \inv{k} 
\quad\text{and}\quad
 d_\text{P}^X(\mu_k^X,\mu^X) \le \inv{k}\cdot
\]
We set $y_{0,k}=y_{0,k}^n=\emptyset^Y$. By \reff{dfpetit}, we get that
for any $k\ge 1,0\le i \le N_k$, there exists $y_{i,k}^n\in Y$ such that
$d^n(x_{i,k},y^n_{i,k}) < \epsilon_n$. Since $Y$ is compact, the
sequence $(y_{i,k}^n,n\ge 1)$ is relatively compact, hence admits a
converging sub-sequence. Using a diagonal argument, and without loss of
generality (by considering the sequence instead of the sub-sequence), we
may assume that for $k\ge 1,0\le i\le N_k$, the sequence
$(y_{i,k}^n,n\ge 1)$ converges to some $y_{i,k}\in Y$.

For any $y\in Y$, we can choose $x\in X$ such that $d^n(x,y)< \epsilon_n$ and
${i,k}$ such
that $d^X(x,x_{i,k})< \inv{k} -\delta$. Then, we get:
\[
 d^Y(y,y_{i,k}^n) = d^n(y,y_{i,k}^n)\le d^n(y,x)+d^X(x,x_{i,k})+ d^n(
 x_{i,k}, y_{i,k}^n)\leq \inv{k}-\delta +2\epsilon_n.
 \]
 Thus, the set $\{ y_{i,k}^n,0\le i \le N_k \}$ is a
 $(2\epsilon_n+1/k-\delta)$-net of $Y$, and the set $S_k^Y = \{ y_{i,k},
 0\le i \le N_k \}$ is an $1/k$-net of $Y$.

If $k,k'\ge 1$
and $0\le i\le N_k,0\le i'\le N_{k'}$, then we have:
\begin{align*} 
d^Y(y_{i,k},y_{i',k'}) 
& \le d^Y(y_{i,k}^n,y_{i,k}) +
d^Y(y_{i,k}^n,y_{i',k'}^n) + d^Y(y_{i',k'}^n,y_{i',k'}) \\
&\le d^Y(y_{i,k}^n,y_{i,k}) + d^Y(y_{i',k'}^n,y_{i',k'}) +2\epsilon_n +
d^X(x_{i,k},x_{i',k'}),
\end{align*}
and, since the terms $d(y_{i,k}^n,y_{i,k})$ and $d(y_{i',k'}^n,y_{i',k'})$ can
be made arbitrarily small, we deduce:
\[
 d(y_{i,k},y_{i',k'}) \le d(x_{i,k},x_{i',k'}).
 \]
The reverse inequality is proven using similar arguments, so that the
above inequality is in fact an equality. Therefore the map defined by $
\Phi(x_{i,k})=(y_{i,k})$ from $\cup_{k \geq 1} S_k$ onto $\cup_{k\geq
 1} S_k^Y$ is a root-preserving 
isometry. By density, this map can be extended uniquely to an isometric
one-to-one root preserving embedding from $X$ to $Y$ which we still
denote by $\Phi$. Hence
the metric spaces $X$ and $Y$ are root-preserving isometric. \\

As far as the measures are concerned, we set:
\begin{equation*}
 \mu_k^{Y,n}=\sum_{i=0}^{N_k} \mu^X(V_{i,k}) \delta_{y_{i,k}^n} 
\quad\text{and}\quad
 \mu_k^Y = \sum_{i=0}^{N_k} \mu^X(V_{i,k}) \delta_{y_{i,k}}.
\end{equation*}
By construction, we have $d_\text{P}^n(\mu_k^{Y,n},\mu_k^X) \le \epsilon_n$. We
get:
\begin{align*}
 d_\text{P}^Y(\mu_k^Y,\mu^Y) = d_\text{P}^n(\mu_k^Y,\mu^Y) 
& \le d_\text{P}^Y(\mu_k^Y,\mu_k^{Y,n})+d_\text{P}^n(\mu_k^{Y,n},\mu_k^X) +
d_\text{P}^X(\mu_k^X,\mu^X)+ d_\text{P}^n(\mu^X,\mu^Y) 
 \\
& < d_\text{P}^Y(\mu_k^Y,\mu_k^{Y,n})+\varepsilon_n+
\inv{k} + \epsilon_n.
\end{align*}
Furthermore, as $n$ goes to infinity, we have that
$d_\text{P}^Y(\mu_k^Y,\mu_k^{Y,n})$ converges to 0,
since the $y^n_{i,k}$ converge towards the $y_{i,k}$. Thus, we actually
have:
\[ 
d_\text{P}^Y(\mu_k^Y,\mu^Y) \le 1/k. 
\]
This implies that $(\mu^Y_k, k\geq 1)$ converges weakly to
$\mu^Y$. Since by definition $\mu_k^Y=\Phi_*\mu_k^X$ and since $\Phi$ is
continuous, by passing to the limit, we get $\mu^Y = \Phi_*\mu^X$. This
gives that $\cx$ and $\cy$ are GHP-isometric.
\end{proof}

This proves that the function $d_{\text{GHP}}^c$ defines a metric on
$\K$. 

\subsection{Proof of Theorem \ref{GHP:PreComK} and of (ii) of 
Theorem \ref{theo:dcGHP}} 
The proof of Theorem \ref{GHP:PreComK} is very close to the proof of
Theorem 7.4.15 in \cite{Burago2001}, where only the Gromov-Hausdorff
metric is involved. It is in fact a simplified version of the proof of
Theorem \ref{GHP:PreComL}, and is thus left to the reader.\\

We are left with the proof of (ii) of Theorem \ref{theo:dcGHP}. It is in
fact enough to check that if $(\cx_n, n\in \N)$ is a Cauchy sequence,
then it is relatively compact.

First notice that if $(Z,d^Z)$ is a Polish metric space, then for any
closed subsets $A,B$, we have $d^Z_\text{P}(A,B)\geq
\left|\diam(A)-\diam(B)\right|$, and for any $\mu,\nu\in \cm_f(Z)$, we
have $d^Z_\text{H}(\mu,\nu)\geq \left|\mu(Z)-\nu(Z)\right|$. This implies
that for any $\cx=(X, d^X, \emptyset^X,\mu), \cy=(Y, d^Y, \emptyset^Y,
\nu) \in \K$: 
\begin{equation}
 \label{eq:diamP}
d^c_\text{GHP}(\cx,\cy)\geq \left|\diam(X)-\diam(Y)\right|+
\left|\mu(X)-\nu(Y)\right|.
\end{equation}
Furthermore, using the definition of the Gromov-Hausdorff metric
\reff{eq:d-GH}, we clearly have:
\begin{equation}
 \label{eq:GHP-GH}
d^c_\text{GHP}(\cx,\cy)\geq d_\text{GH}^c((X,d^X), (Y, d^Y)).
\end{equation}

We deduce that if $\ca=(\cx_n, n\in \N)$ is a Cauchy sequence, then
\reff{eq:diamP} implies that conditions (i) and (iii) of Theorem
\ref{GHP:PreComK} are fulfilled. Furthermore, thanks to
\reff{eq:GHP-GH}, the sequence $((X_n, d^{X_n}), n\in \N)$ is a Cauchy
sequence for the Gromov-Hausdorff metric. Then point (2) of
Proposition 7.4.11 in \cite{Burago2001} readily implies condition (ii)
of Theorem \ref{GHP:PreComK}.

\section{Extension to locally compact length spaces}
\label{sec:L}
\subsection{First results}

First, let us state two elementary Lemmas. Let $(X,d, \emptyset)$ be a
rooted metric space. Recall notation \reff{eq:X(r)}. We set:
\[
\partial_r X= \{x\in X; \ d(\emptyset^x,x)=r\}.
\]

\begin{lemma}
\label{lem:couronne}
Let $(X, d,\emptyset)$ be a complete rooted length space and
$r,\epsilon>0$. Then we have, for all $\delta>0$:
\[ X^{(r+\epsilon)} \subset (X^{(r)})^{\epsilon+\delta}. \]
\end{lemma}
\begin{proof}
 Let $x\in X^{(r+\epsilon)}\backslash X^{(r)}$ and $\delta>0$. There
 exists a rectifiable curve $\gamma$ defined on $[0,1]$ with values in
 $X$ such that $\gamma(0)=\emptyset$ and $\gamma(1)=x$, such that
 $L(\gamma) < d(\emptyset,x)+\delta\le r+\epsilon +\delta$. There
 exists $t\in (0,1)$ such that $\gamma(t)\in \partial_r X$. We can
 bound $d(\gamma(t),x)$ by the length of the fragment of $\gamma$
 joining $\gamma(t)$ and $x$, that is the length of $\gamma$ minus the
 length of the fragment of $\gamma$ joining $\emptyset$ to
 $\gamma(t)$. The latter being equal to or larger than $d(\emptyset^X,
 \gamma(t))=r$, we get:
\[
d(\gamma(t),x) \leq L(\gamma) - r< \varepsilon+\delta. 
\]
Since $\gamma(t)\in X^{(r)}$, we get $x\in
\left(X^{(r)}\right)^{\varepsilon+\delta}$. This ends the proof.
\end{proof}

\begin{lemma}\label{df:estimation}
 Let $\cx=(X,d,\emptyset,\mu) \in \LL$. For all $\varepsilon>0$ and
 $r>0$, we have:
\[
 d_{\text{GHP}}^c(\cx^{(r)},\cx^{(r+\epsilon)}) \le \epsilon +
\mu(X^{(r+\epsilon)}\setminus X^{(r)}). 
\] 
\end{lemma}
\begin{proof}
 The identity map is an obvious embedding $X^{(r)}\hookrightarrow
 X^{(r+\epsilon)}$ which is root-preserving. Then, we have:
\[
 d_{\text{GHP}}^c(\cx^{(r)},\cx^{(r+\epsilon)}) \le
d_\text{H}(X^{(r)},X^{(r+\epsilon)}) +
d_\text{P}(\mu^{(r)},\mu^{(r+\epsilon)}). 
\]
Thanks to Lemma \ref{lem:couronne}, we have
$d_\text{H}(X^{(r)},X^{(r+\epsilon)})
\le \epsilon$. 

Let $A\subset X$ be closed. We have obviously $\mu^{(r)}(A)\le
\mu^{(r+\epsilon)}(A)$. On the other hand, we have:
\[ 
\mu^{(r+\epsilon)}(A) \le \mu^{(r)}(A) + \mu(A\cap
(X^{(r+\epsilon)}\setminus X^{(r)})) \le \mu^{(r)}(A) +
\mu(X^{(r+\epsilon)}\setminus X^{(r)}). 
\]
This proves that $d_\text{P}(\mu^{(r)},\mu^{(r+\epsilon)})\le
\mu(X^{(r+\epsilon)}\setminus X^{(r)})$, which ends the proof.
\end{proof}

It is then straightforward to prove Lemma \ref{lem:reg-d-GHP}. 

\begin{proof}[Proof of Lemma \ref{lem:reg-d-GHP}]
 Let $\cx=(X,d^X,\emptyset^X, \mu^X)$ and $\cy=(Y, d^Y, \emptyset^Y,
 \mu^Y)$ be two elements of $\LL$. Using the triangular inequality
 twice and Lemma \ref{df:estimation}, we get for $r>0$ and $\epsilon
 >0$:
\begin{align*}
 |d_{\text{GHP}}^c(\cx^{(r)},\cy^{(r)})-d_{\text{GHP}}^c(\cx^{(r+\epsilon)},\cy^
{(r+\epsilon)}
) | 
& \le d_{\text{GHP}}^c(\cx^{(r)},\cx^{(r+\epsilon)}) +
d_{\text{GHP}}^c(\cy^{(r)},\cy^{(r+\epsilon)}) \\
& \le 2\epsilon + \mu^X(X^{(r+\epsilon)}\setminus X^{(r)}) +
\mu^Y(Y^{(r+\epsilon)}\setminus Y^{(r)}).
\end{align*}
As $\epsilon$ goes down to 0, the expression above converges to 0, so that we
get
right-continuity of the function $r\mapsto
d_{\text{GHP}}^c(\cx^{(r)},\cy^{(r)})$.

We write $\cx^{(r-)}$ for the compact metric space $X^{(r)}$ rooted at
$\emptyset^X$ along with the induced metric and the restriction of $\mu$
to the \textit{open} ball $\{x\in X; \ d^X(\emptyset^X,
x)<r\}$. We define $\cy^{(r-)}$ similarly. Similar arguments as above
yield for $r>\varepsilon>0$:
\begin{multline*}
 |d_{\text{GHP}}^c(\cx^{(r-)},\cy^{(r-)})-d_{\text{GHP}}^c(\cx^{(r-\epsilon)},
\cy^{(r-\epsilon) } ) | \\
\begin{aligned}
 &\le d_{\text{GHP}}^c(\cx^{(r-)},\cx^{(r-\epsilon)}) +
d_{\text{GHP}}^c(\cy^{(r)},\cy^{(r-\epsilon)}) \\
& \le 2\epsilon + \mu^X(\{x \in X,\
r-\varepsilon<d^X(\emptyset^X,x)<r\})+
\mu^Y(\{y \in Y,\ r-\varepsilon<d^Y(\emptyset^Y,y)<r\}).
\end{aligned}
\end{multline*}
As $\epsilon$ goes down to 0, the expression above also converges to 0,
which shows the existence of left limits for the function $r\mapsto
d_{\text{GHP}}^c(\cx^{(r)},\cy^{(r)})$.
\end{proof}

The next result corresponds to (i) in Theorem \ref{theo:LL}. 

\begin{proposition}\label{prop:distance}
 The function $d_{\text{GHP}}$ is a metric on $\LL$. 
\end{proposition}

\begin{proof}
 The symmetry and positivity of $d_{\text{GHP}}$ are obvious. The triangle
 inequality is not difficult either, since $d_{\text{GHP}}^c$ satisfies the
 triangle inequality and the map $x\mapsto 1\wedge x$ is non-decreasing
 and sub-additive.

 We need to check that $d_{\text{GHP}}$ is definite positive. To that
 effect, let $\cx=(X,d^X,\emptyset^X,\mu)$ and
 $\cy=(Y,d^Y,\emptyset^Y,\nu)$ be two elements of $\LL$ such that
 $d_{\text{GHP}}(\cx,\cy)=0$. We want to prove that $\cx$ and $\cy$
 are GHP-isometric. We follow the spirit of the proof of Lemma
 \ref{lem:dc=0}.

 By definition, we get that for almost every $r>0,\
 d_{\text{GHP}}^c(\cx^{(r)},\cy^{(r)})=0$. Let $(r_n,\ n\ge 1)$ be a
 sequence such that $r_n \uparrow \infty$ and such that for $n\ge 1,\
 d_{\text{GHP}}^c(\cx^{(r_n)},\cy^{(r_n)})=0$. Since
 $d_{\text{GHP}}^c$ is a metric on $\K$, there exists a GHP-isometry
 $\Phi^{n}: X^{(r_n)} \rightarrow Y^{(r_n)}$ for every $n\ge 1$. Since
 all the $X^{(r)}$ are compact, we may consider, for $n\ge 1$ and for
 $k\ge 1$, a finite $1/k$-net of $X^{(r_n)}$ containing the root:
\[
 S_k^{n} = \{ x_{0,k}^n=\emptyset^X, x^{n}_{1,k},...,x^{n}_{N^{n}_k,k} \}. 
\]
Then, if $k\ge 1$, $n\ge 1$, $0\le i \le N^{n}_k$, the sequence
$(\Phi^{j}(x^{n}_{i,k}),\ j\ge n)$ is bounded since the $\Phi^{j}$ are
isometries. Using a diagonal procedure, we may assume without loss of
generality, that for every $k\ge 1$, $n\ge 1$, $0\le i \le N^{n}_k$, the
sequence $(\Phi^{j}(x^{n}_{i,k}),\ j\ge n)$ converges to some
limit $y^{n}_{i,k}\in Y$. We define the map $\Phi$ on $S:=\bigcup_{n\ge 1,\
 k\ge 1} S^{n}_k $ taking values in $Y$ by:
\[ 
\Phi(x^{n}_{i,k}) = y^{n}_{i,k}.
\]
Notice that $\Phi$ is an isometry and root preserving as
$\Phi(\emptyset^X)=\emptyset^Y$ (see the proof of Lemma
\ref{lem:dc=0}). The set $\Phi(S^n_k)$ is obviously a $2/k$-net of
$Y^{(r_n)}$, and thus $\Phi(S)$ is a dense subset of $Y$. Therefore the
map $\Phi$ can be uniquely extended into a one-to-one root preserving
isometry from $X$ to $Y$, which we shall still denote by $\Phi$. It
remains to prove that $\Phi$ is a GHP-isometry, that is, such that
$\nu=\Phi_*\mu$.

For $n\ge 1$, $k\ge 1$, let $(V^{n}_{i,k},\ 0\le i\le N^{n}_k)$ be Borel
subsets of $X^{(r_n)}$ with diameter less than $1/k$, such that $\bigcup
_{0\le i\le N_k} V^n_{i,k}=X^{(r_n)}$ and for all $0\le i,i'\le N_k$, we
have $V_{i,k}^n\bigcap V_{i',k}^n=\emptyset$ and $x_{i,k}^n\in
V_{i,k}^n$ if $V_{i,k}^n\neq \emptyset$. We then define the following
measures:
\[
 \mu^{n}_k = \sum_{i=0}^{N^{n}_k} \mu(V^{n}_{i,k})
\delta_{x^{n}_{i,k}}
\quad\text{and}\quad
 \nu^{n}_{k} = \sum_{i=0}^{N^{n}_k} \mu(V^{n}_{i,k})
\delta_{y^{n}_{i,k}}.
\]
Let $A\subset X$ be closed. We obviously have $\mu_k^n(A)\leq
\mu^{(r_n)} (A^{1/k})$ and $\mu^{(r_n)} (A) \leq \mu^n_k(A^{1/k})$ that is:
\begin{equation}
 \label{eq:distPmu}
d_\text{P}^X( \mu^n_k,\mu^{(r_n)})\leq \inv{k}\cdot
\end{equation}

For any $n\ge 1$, $k\ge 1$, we have by construction $\nu^{n}_{k} =
\Phi_*\mu_k^{n}$ and $\nu^{(r_n)}=\Phi^{j}_*\mu^{(r_n)}$ for any $j\ge
n\ge 1$. We can then write, for $j\ge n$:
\begin{align*}
 d_\text{P}^{Y}(\nu^{n}_k,\nu^{(r_n)}) 
& = d_\text{P}^{Y}(\Phi_*\mu^{n}_k,\Phi^{j}_*\mu^{(r_n)} ) \\
& \le d_\text{P}^{Y}(\Phi_*\mu^{n}_k,\Phi^{j}_*\mu^{n}_k) +
d_\text{P}^{Y}(\Phi^{j}_* \mu^{n}_k, \Phi^{j}_* \mu^{(r_n)})\\
&\leq d_\text{P}^{Y}(\Phi_*\mu^{n}_k,\Phi^{j}_*\mu^{n}_k) +\inv{k},
\end{align*}
where for the last inequality we used $d_\text{P}^{Y}(\Phi^{j}_* \mu^{n}_k,
\Phi^{j}_* \mu^{(r_n)})= d_\text{P}^{X}( \mu^{n}_k,\mu^{(r_n)})$ and
\reff{eq:distPmu}. Since the two measures $\Phi_*\mu^{n}_k$ and
$\Phi^{j}_*\mu^{n}_k$ have the same masses distributed on a finite
number of atoms, and the atoms $\Phi^{j}(x^{n}_{i,k})$ of
$\Phi^{j}_*\mu^{n}_k$ converge towards the atoms $y^{n}_{i,k}$ of
 $\Phi_*\mu^{n}_k$, we deduce that:
\[
\lim_{j\rightarrow+\infty }
 d_\text{P}^{Y}(\Phi_*\mu^{n}_k,\Phi^{j}_*\mu^{n}_k) =0. 
\]
Hence, $(\nu^{n}_k, k\geq 1)$ converges weakly towards $\nu^{(r_n)}$.
According to \reff{eq:distPmu}, the sequence $(\mu_k^n, k\geq 1)$
converges weakly to $\mu^{(r_n)}$. Since we have
$\nu^{n}_k=\Phi_*\mu^{n}_k$ and $\Phi$ is continuous, we get
$ \nu^{(r_n)}=\Phi_*\mu^{(r_n)} $ for any $n\ge 1$, and thus
$\nu=\Phi_*\mu$. This ends the proof.
\end{proof}

We are now ready to prove Proposition \ref{prop:K-LL}. Notice that we
shall not use (ii) of Theorem \ref{theo:LL} in this Section as it is not
yet proved.

\begin{proof}[Proof of Proposition \ref{prop:K-LL}]
By construction, the convergence in $\K\cap \LL$ for the 
$d_{\text{GHP}}$ metric implies the convergence for the $d_{\text{GHP}}^c$
metric. We
only have to prove that the converse is also true.

Let $\cx=(X,d^X, \emptyset, \mu)$ and $\cx_n=(X_n, d^{X_n}, \emptyset_n,
\mu_n)$ be elements of $\K\cap \LL$ and $(\epsilon_n, n\in \N)$ be a
positive sequence converging towards 0 such that, for all $n\in \N$:
\[ 
 d_{\text{GHP}}^c(\cx_n,\cx) <\epsilon_n .
\]
Using Lemma \ref{df:uniondisjointe}, we consider a metric $d^n$ on the disjoint
union $X_n\sqcup X$, such that we have for $n\in \N$, and writing
$d^n_\text{H}$ and $d^n_\text{P}$ respectively for $d^{d^n}_\text{H}$ and
$d^{d^n}_\text{P}$:
\[ 
d^n(\emptyset_n,\emptyset)+ d_\text{H}^n(X_n,X) + d_\text{P}^n(\mu_n,\mu) <
\epsilon_n. 
\]
If $x_n\in X_n^{(r)}$, by definition of the Hausdorff metric, there exists
$x\in X$ such that $d^n(x_n,x)\leq d_\text{H}^n(X_n,X)$. Then, we have:
\[
 d^n(\emptyset,x) \le d^n(\emptyset, \emptyset_n) + d^n(\emptyset_n,x_n) +
d^n(x_n,x) 
 \le d^n(\emptyset_n,\emptyset) + r + d_\text{H}^n(X_n,X) 
 < r+\epsilon_n.
\]
We get that $x$ belongs to $X^{(r+\epsilon_n')}$ for some
$\varepsilon'_n<\varepsilon_n$ and thus, according to Lemma
\ref{lem:couronne}, it belongs to $(X^{(r)})^{\epsilon_n}$, since $X$ is
a complete length space. Therefore we have $X_n^{(r)}\subset
(X^{(r)})^{\varepsilon_n}$. Similar arguments yield $X^{(r)}\subset
(X_n^{(r)})^{\varepsilon_n}$. We deduce that:
\begin{equation}\label{dHaus}
d_\text{H}^n(X_n^{(r)},X^{(r)}) \le\epsilon_n. 
\end{equation}

If $A\subset X_n\sqcup X$ is closed, we may compute:
\begin{align*}
 \mu_n^{(r)}(A) = \mu_n(A\cap X_n^{(r)}) 
& \le \mu( A^{\epsilon_n} \cap
(X_n^{(r)})^{\epsilon_n}) +\epsilon_n \\
& \le \mu^{(r)}(A^{\epsilon_n}) + \mu((X_n^{(r)})^{\epsilon_n} \setminus
X^{(r)}) + \epsilon_n \\
& \le \mu^{(r)}(A^{\epsilon_n}) + \mu(X^{(r+2\epsilon_n)} \setminus X^{(r)})
+\epsilon_n, 
\end{align*}
since $(X_n^{(r)})^{\varepsilon_n}\subset
(X^{(r)})^{2\varepsilon_n} \subset X^{(r+2\varepsilon_n)}$. Similarly,
we also have: 
\begin{align*}
 \mu^{(r)}(A) & 
\le \mu(A\cap X^{(r-2\epsilon_n)}) + \mu(X^{(r)}\setminus
X^{(r-2\epsilon_n)}) \\
& \le \mu_n(A^{\epsilon_n} \cap (X^{(r-2\epsilon_n)})^{\epsilon_n} ) +
\mu(X^{(r)}\setminus X^{(r-2\epsilon_n)}) +\epsilon_n\\
& \le \mu_n^{(r)}(A^{\epsilon_n} ) +
\mu(X^{(r)}\setminus X^{(r-2\epsilon_n)}) +\epsilon_n,
\end{align*}
since $(X_n^{(r-2\epsilon_n)})^{\epsilon_n} \subset
 X^{(r)}$. Hence, we
finally deduce:
\[
d_\text{P}^n(\mu_n^{(r)},\mu^{(r)}) \le \epsilon_n +
\mu(X^{(r+2\epsilon_n)} \setminus X^{(r-2\epsilon_n)}). 
\]
This and \reff{dHaus} yield:
\[ 
d_{\text{GHP}}^c(\cx_n^{(r)},\cx^{(r)}) \le 3 d_{\text{GHP}}^c(\cx_n,\cx) + \mu(
X^{(r+2\epsilon_n)}
\setminus X^{(r-2\epsilon_n)}). 
\]
Therefore, if $\mu(\partial_r X) = 0$, we have
$\lim_{n\rightarrow+\infty } d_{\text{GHP}}^c(\cx_n^{(r)},\cx^{(r)})=0$. Since
$\mu$
is by definition a finite measure, the set $\{r>0,\ \mu(\partial_r X) \ne
0\}$ is at most countable. By
dominated convergence, we get $\lim_{n\rightarrow+\infty } 
d_{\text{GHP}}(\cx_n,\cx)=0$.
\end{proof}

In order to prove Theorem \ref{GHP:PreComL} on the pre-compactness
criterion, we shall approximate the elements 
of a sequence in $\cc$ by nets of small radius. The following Lemma guarantees
that we can construct such nets in a consistent way. We use the
convention that $X^{(r)}=\emptyset$ if $r<0$. In the sequel, if $r>0$ and $k\ge
0$, we will often use the notation $A_{r,k}(X)$ for the annulus
$X^{(r)}\setminus X^{(r-2^{-k})}$.

\begin{lemma}\label{KNetCouronne}
If $\cx=(X,\emptyset,d,\mu)\in \LL$ satisfies condition (i) of
Theorem \ref{GHP:PreComL}, then for any $k, \ell\in \N$, there
exists a $2^{-k}$-net of the annulus $A_{\ell2^{-k},k}(X) =
X^{(\ell 2^{-k})} \setminus X^{((\ell-1)2^{-k})}$ with at
most $N(\ell2^{-k},2^{-k-1})$ elements. 
\end{lemma}
\begin{proof}
 Let $S'$ be a finite $2^{-k-1}$-net of $X^{(\ell 2^{-k})}$ of cardinal
 at most $N(\ell2^{-k},2^{-k-1})$. Let $S''$ be the set of elements $x$
 in $S'\cap A_{(\ell-1)2^{-k},k+1}(X)$ such that there exists at least
 one element, say $y_x$, in $A_{\ell2^{-k},k}(X)$ at distance at most
 $2^{-k-1}$ of $x$. The set $\left(S'\cap A_{\ell2^{-k},k}\right)
 \bigcup \{y_x, \ x\in S''\}$ is obviously a $2^{-k}$-net of
 $A_{\ell2^{-k},k}(X)$, and its cardinal is bounded by
 $N(\ell2^{-k},2^{-k-1})$. 
\end{proof}

\subsection{Proof of Theorem \ref{GHP:PreComL}}
Notice that we shall not use (ii) of Theorem \ref{theo:LL} in this
Section as it is not yet proved.

The proof will be divided in several parts. The idea, as in
\cite{Burago2001}, is to construct an abstract limit space, along with a
measure, and to check that we can get a convergence (up to
extraction). Let $(\cx_n, n\in \N)$ be a sequence in $\cc$, with
$\cx_n=(X_n, d^{X_n}, \emptyset_n, \mu_n)$. For $\ell,k\in \N$, we will
write $\ell_k$ for $\ell 2^{-k}$. 

\subsubsection{Construction of the limit space.} 
Let $\ell,k\in \N$. Recall that, by Lemma \ref{KNetCouronne}, we can
consider $\as_{\ell_k,k}^n$ a $2^{-k-1}$-net of the annulus
$A_{\ell_k,k}(X_n)$ with at most $N(\ell_k,2^{-k-2})$ elements. In
order to have a finer sequence of nets, we shall consider:
\[
S_{\ell_k,k}^n=\bigcup _{0\leq k'\leq k} \left(A_{\ell_k,k}(X_n)\cap
\as_{\lceil \ell_k 2^{k'} \rceil 2^{-k'},k'}^n\right).
\]
By construction $S^{n}_{\ell_k,k} $ is a $2^{-k-1}$-net of
$A_{\ell_k,k}(X_n)$ with cardinal at most:
\[
\bar N(\ell_k,2^{-k-2})=
\sum_{k'=0}^k N(\lceil \ell_k 2^{k'} \rceil 2^{-k'},2^{-k'-2}).
\]
Let
$U_{\ell_k,k}=\{(k,\ell,i); 0\le i\le \bar N(\ell_k,2^{-k-2}) \}$
and $U=\bigcup 
_{k\in \N, \ell\in \N} U_{\ell_k,k}$. 
We number the elements
of $S_{\ell_k,k}^n$ in such a way that:
\begin{equation}
 S_{\ell_k,k}^n \cup \{\emptyset_n\} = \{ x^n_u,\
u=(k,\ell,i), u\in U_{\ell_k, k} \}, 
\end{equation}
where $(x^n_u, u\in U)$ is some
sequence in $X_n$ and $x^n_{(k,\ell,0)}=\emptyset_n$. Notice that
$S_{\ell_k,k}^n$ is empty for $\ell_k$ large if $X_n$ is bounded. For
$u, u'\in U$, we set:
\[
 d_{u,u'}^n = d^{X_n}(x^n_u,x^n_{u'}). 
\]
Notice that the sequence $(d_{u,u'}^n, n\in \N)$ is bounded. Thus, without
loss of generality (by considering the sequence instead of the
sub-sequence), we may assume that for all $u,u'\in U$, the sequence
$(d^n_{u,u'}, n\ge 1)$ converges in $\R$ to some limit $d_{u,u'}$. We
then consider an abstract space, $X'=\{ x_u, u\in U\}$. On this space,
the function $d$ defined by $(x_u,x_{u'})\mapsto d_{u,u'}$ is a
semi-metric. We then consider the quotient space $X'{/\sim}$, where
$x_u\sim x_{u'}$ if $d_{u,u'}=0$. We shall denote by $x_u$ the
equivalent class containing $x_u$. Notice that $d_{u,u'}=0$ for any
$u=(k,\ell,0)$ and $u'=(k', \ell', 0)$ elements of $U$ and let
$\emptyset$ denote their equivalence class. Finally, we let $X$ be the
completion of $X'{/\sim}$ with respect to the metric $d$, so that
$(X,d,\emptyset)$ is a rooted complete metric space.

\subsubsection{Approximation by nets}

We set:
\[
U_{\ell_k,k}^+=\bigcup _{0\leq j\leq \ell} U_{j2^{-k},k},
\quad
 S^{n,+}_{\ell_k,k} = \bigcup _{0\leq j\leq \ell} S^n_{j2^{-k}, k}=
 \{x^n_u, u\in U_{\ell_k,k}^+\} 
\quad\text{and}\quad
 S^{+}_{\ell_k,k} = \{x_u, u\in U_{\ell_k,k}^+\}.
\]
By construction $S^{n,+}_{\ell_k,k} $ is a $2^{-k-1}$-net of
$X_n^{(\ell_k)}$ and $S^{n,+}_{\ell_k,k} \subset S^{n,+}_{\ell'_{k'},k'} $
as well as $S^{+}_{\ell_k,k} \subset S^{+}_{\ell'_{k'},k'} $ for any
$k\leq k'$ and $\ell_k\leq \ell'_{k'}$. 

\begin{rem}
 \label{rem:stricte}
 We also have that for $v\in U\backslash U^{+}_{\ell_k,k}$, either
 $x^n_v=\emptyset_n$ or $d^{X_n}(\emptyset_n, x^n_v)>\ell_k$ and
 either $x_v=\emptyset$ or $d(\emptyset, x_v)\geq \ell_k$. Notice
 that the
 former inequality is strict but the latter is large.
\end{rem}

A correspondence $R$ between two sets $A$ and $B$ is a subset of
$A\times B$ such that the projection of $R$ on $A$ (resp. $B$) is $A$
(resp. $B$). It is clear that the set defined by: 
\begin{equation}
\label{eq:correspondence}
{\crr}_{\ell_k,k}^{n,+} = \{ (x^n_u,{x}_u), u\in U_{\ell_k,k}^+\}
\end{equation}
is a correspondence between
$S^{n,+}_{\ell_k,k} $ and $S^{+}_{\ell_k,k} $.
The distorsion $\delta_n(\ell_k,k)$ 
of this correspondence is defined by:
\begin{equation}
\label{deltan} 
\delta_n(\ell_k,k) = \sup\{
|d^{X_n}(x_u^n,x_{u'}^n)-d(x_u,x_{u'})|; \ u,u'\in U_{\ell_k,k}^+\}. 
\end{equation}
Notice that for $k\leq k'$ and $\ell_k\leq \ell'_{k'}$, we have:
\begin{equation}
 \label{eq:ineg-delta}
\delta_n(\ell_k,k)\leq \delta_n(\ell'_{k'}, k').
\end{equation}
Since $U_{\ell_k, k}^+$ is finite, for all $\ell, k\in \N$, we have by
construction $\lim_{n\rightarrow+\infty } \delta_n(\ell_k,k) =0$.

\begin{lem}
 \label{lem:approx}
The set $S^{+}_{\ell_k,k} $ is a $2^{-k}$-net of
$X^{(\ell_k)}$. 
\end{lem}
\begin{proof}
 Let $x\in X^{(\ell_k)}$. There exists $v=(k',\ell',j)\in U$ such that
 $d(x,x_v)<2^{-k-3}$. Notice that $d(\emptyset, x_v) < \ell_k+
 2^{-k-3}$. We may choose $n$ large enough, so that $\delta_n(\ell_k
 \vee \ell'_{k'} ,k\vee k')<2^{-k-3}$. As $x^n_v\in S^{n,+}_{\ell_k
 \vee \ell'_{k'} ,k\vee k'}$, we have
 $|d^{X_n}(\emptyset_n, x^n_v) -d(\emptyset, x_v)|< 2^{-k-3}$ and thus
 $d^{X_n}(\emptyset_n, x^n_v)<\ell_k+ 2^{-k-2}$. Thanks to Lemma
 \ref{lem:couronne} and since $X_n$ is a length space, we get that
 $x^n_v$ belongs to $(X_n^{(\ell_k)})^{2^{-k-2}}$. As
 $S^{n,+}_{\ell_k,k} $ is a $2^{-k-1}$-net of $X_n^{(\ell_k)}$, there
 exists $u\in U_{\ell_k,k}^+$ such that $d^{X_n}(x^n_u , x^n_v)<
 2^{-k-1}+2^{-k-2}$. Furthermore, we have that $x^n_u$ and $x^n_v$
 belongs to $S^{n,+}_{\ell_k
 \vee \ell'_{k'} ,k\vee k'}$. We deduce that:
\[
d(x,x_u)
\leq d(x,x_v)+d(x_v,x_u)
\leq 2^{-k-3} + \delta_n(\ell_k
 \vee \ell'_{k'} ,k\vee k') + d^{X_n} (x^n_u , x^n_v) < 2^{-k}.
\]
This gives the result.
\end{proof}

We give an immediate consequence of this approximation by nets.
\begin{lem}
 \label{lem:X=length-space}
The metric space $(X,d)$ is a length space. 
\end{lem}
\begin{proof}
 The proof of this Lemma is inspired by the proof of Theorem 7.3.25 in
 \cite{Burago2001}. We shall check that $(X,d)$ satisfies the mid-point
 condition.

 Let $k\in \N$ and $x,x'\in X$. According to Lemma \ref{lem:approx},
 there exists $\ell\in \N$ large enough and $u,u'\in U^+_{\ell_k, k}$
 such that $d(x,x_u)<2^{-k}$ and $d(x',x_{u'})<2^{-k}$. For $n$ large
 enough, we get that $\delta_n(\ell_k,k) <2^{-k}$. Since $(X_n,
 d^{X_n})$ is a length space, there exists $z\in X_n$ such that:
\[
|2d^{X_n}(z,x^n_u)- d^{X_n}(x^n_u,x^n_{u'})|+ 
|2d^{X_n}(z,x^n_{u'})- d^{X_n}(x^n_u,x^n_{u'})|\leq 2^{-k}.
\]
There exists $u''\in U^+_{\ell_k, k}$ such that $d^{X_n}(x^n_{u''}, z) \leq
2^{-k}$. Then, we deduce that:
\begin{multline*}
|2d(x_{u''},x)- d(x,x')|+ 
|2d(x_{u''},x')- d(x,x')|\\ 
\begin{aligned}
& \leq 
4d(x,x_u)+4d(x',x_{u'})+ 
|2d(x_{u''},x_u)- d(x_u,x_{u'})|+ 
|2d(x_{u''},x_{u'})- d(x_u,x_{u'})|\\
& \leq 8. 2^{-k} + 6 \delta_n(\ell_k,k) + |2d^{X_n}(x^n_{u''},x^n_u)- 
d^{X_n}(x^n_u,x^n_{u'})|+ 
|2d^{X_n}(x^n_{u''},x^n_{u'})- d^{X_n}(x^n_u,x^n_{u'})|\\ 
& \leq 19. 2^{-k}. 
\end{aligned}
\end{multline*}
Since $k$ is arbitrary, we get that $(X,d)$ satisfies the mid-point
condition and thus is a length space.
\end{proof}

\subsubsection{Approximation of the measures}
\label{sec:approx-meas}
Let $(V^n_u, u\in U_{\ell_k,k})$ be Borel subsets of
$A_{\ell_k,k}(X_n)$ with diameter
less than $2^{-k}$ 
such that $\bigcup _{u\in U_{\ell_k,k}} V^n_u= A_{\ell_k,k}(X_n)$ 
 and 
for all $u,u'\in U_{\ell_k,k}$, we have $V^n_u\bigcap
V^n_{u'}=\emptyset$ and $x^n_u\in V^n_u$ as soon as $V^n_u\neq
\emptyset$. We set $U_{\infty ,k}=\bigcup _{ \ell\in \N} U_{\ell_k,k}$
and we consider the following
approximation of the measure $\mu_n$:
\[
 \mu_{n,k} = \sum _{u\in U_{\infty ,k}} \mu_n(V^n_u) \delta_{x_u^n}.
\]
Notice that $ \mu_{n,k}^{(\ell_k)} = \sum _{u\in U_{\ell_k ,k}} \mu_n(V^n_u)
 \delta_{x_u^n}$. 
The measures $\mu_{n,k}$ are locally finite Borel measures on $X_n$. It
is clear that the sequence $(\mu_{n,k}, k\in \N)$ converges 
vaguely towards $\mu_n$
as $k$ goes to infinity, since we have for any $r\in \N$,
$d_\text{P}^{d^{X_n}}(\mu_{n,k}^{(r)},\mu_n^{(r)})\le 2^{-k}$. 
On the limit space $X$, we define:
\[
 \nu_{n,k} = \sum_{u\in U_{\infty ,k}}
\mu_n(V^n_u) \delta_{x_u}
\quad\text{and}\quad
 \nu_{n,k}^{\{\ell_k\}} = \sum_{u\in U_{\ell_k ,k}}
\mu_n(V^n_u) \delta_{x_u}. 
\]
Notice that $\nu_{n,k}^{\{\ell_k\}}\leq \nu_{n,k}^{(\ell_k)} $ but they
may be distinct as $\nu_{n,k}^{(\ell_k)} $ may have some atoms on
$\partial_{\ell_k} X$ which are in $S_{(\ell+1)_k,k}^+$ but not in
$S_{\ell_k,k}^+$, as indicated in Remark \ref{rem:stricte}.

Let us show that the sequence $(\nu_{n,k}, k\in \N)$ converges, up to an
extraction, towards a locally finite measure $\nu$ on $X$. For $m \in
2^{-k} \N$, we
have:
\begin{align}
\nonumber
 \nu_{n,k}(X^{(m)}) 
= \sum_{u\in U_{\infty ,k}}
\mu_n(V^n_u) \ind_{\{d(x_u,\emptyset)\le m\}} 
& \le \sum_{u\in U_{\infty ,k}}
\label{eq:majo-mu-nu}
\mu_n(V^n_u) \ind_{\{d^{X_n} (x_u^n,\emptyset_n)\le m+\delta_n(m,k)\}}\\
& \le \mu_n(X_n^{(m+\delta_n(m,k)+2^{-k})}),
\end{align}
where for the first inequality we used \reff{deltan}. 
Recall that
for all $\ell, k\in \N$, we have $\lim_{n\rightarrow+\infty }
\delta_n(\ell_k,k) =0$. We define $\eta_k= \delta_{n_k}(k,k) $. 
Using a diagonal argument, there exists a
sub-sequence $(n_k, k\in \N)$ such that:
\begin{equation}
 \label{eq:eta}
\eta_k\leq 2^{-k}.
\end{equation}
By \reff{eq:ineg-delta}, we have $\delta_{n_k} (m,k)\leq \eta_k$ for
$k\geq m$. Thanks to property (ii) of Theorem \ref{GHP:PreComL}, we get
that $\mu_{n_k}{(X_{n_k})^{(m+\delta_{n_k}(m,k)+2^{-k})}}$ is uniformly
bounded in $k\in \N$ for $m$ fixed. From the classical pre-compactness
criterion for vague convergence of locally finite measures on a Polish
metric space (see Appendix 2.6 of \cite{Daley2003}), we deduce that there
exists an extraction of the sub-sequence $(n_k, k\in \N)$, which we still
note $(n_k, k\in \N)$, such that $(\nu_{n_k,k}, k\in \N)$ converges
vaguely towards some locally finite measure $\nu$ on $X$. This implies
 the weak convergence of the finite measures
$(\nu_{n_k,k}^{(r)}, k\in \N)$ towards $\nu^{(r)}$ as soon as
$\nu(\partial_r X)=0$. Since $\nu$ is locally finite, the set 
\begin{equation}
 \label{eq:A-nu}
A_\nu=\{r\geq 0; \ \nu(\partial_r X)>0\}
\end{equation}
is at most countable. Thus, we have
$\lim_{n\rightarrow+\infty } d_\text{P} (\nu_{n_k,k}^{(r)}
,\nu^{(r)})=0$ for almost every $r>0$.

\subsubsection{Convergence in the $d_{\text{GHP}}$ metric.} 
We set $\cx=(X,d,\emptyset, \nu)$. Notice that $\cx\in \LL$ thanks to
Lemma \ref{lem:X=length-space}. We shall
prove that $d_{\text{GHP}}(\cx_{n_k},\cx )$ converges to $0$.\\

Let $r>0$. For any $k\in \N$, set $\ell=\lceil 2^k r \rceil$ and recall
$\ell_k=2^{-k}\lceil 2^k r \rceil$. We set:
\[
\cy^n_k=(S_{\ell_k,k}^{n,+}, d^{X_n}, \emptyset_n, \mu_{n,k}^{(\ell_k)}),
\quad
\cz^n_k=(S_{\ell_k,k}^{+}, d, \emptyset, \nu_{n,k}^{\{\ell_k\}})
\quad\text{and}\quad
\cw_k^n=(X^{(\ell_k)}, d, \emptyset, \nu_{n,k}^{\{\ell_k\}}).
\]
 The
triangular inequalities give:
\begin{equation}
 \label{eq:majo-dghpc}
 d_{\text{GHP}}^c(\cx_n^{(r)},\cx^{(r)})
\leq B_n^1+B_n^2+B_n^3+B_n^4+B_n^5+B_n^6,
\end{equation}
with:
\begin{align*}
B_n^1
&= d_{\text{GHP}}^c\left(\cx_n^{(r)}, \cx_n^{(\ell_k)}\right),
 \quad
B_n^2
= d_{\text{GHP}}^c\left(\cx_n^{(\ell_k)}, \cy^n_k \right) ,
\quad
B_n^3
= d_{\text{GHP}}^c\left(\cy^n_k, \cz_k^n \right) ,\\
B_n^4
&= d_{\text{GHP}}^c\left(\cz_k^n ,\cw^n_k\right) ,
\quad
B_n^5
= d_{\text{GHP}}^c \left( \cw^n_k, \cx^{(\ell_k)}\right) , 
\quad
B_n^6
=d_{\text{GHP}}^c \left( \cx^{(\ell_k)}, \cx^{(r)}\right).
\end{align*}

Lemma \ref{df:estimation} implies that:
\begin{equation}
 \label{eq:B1}
B_n^1= d_{\text{GHP}}^c\left(\cx_n^{(r)}, \cx_n^{(\ell_k)}\right) \leq
2^{-k} + \mu_n(X_n^{(\ell_k)}\setminus X_n^{(r)}). 
\end{equation}

As $S^{n,+}_{\ell_k,k} $ is a $2^{-k-1}$-net of
$X_n^{\ell_k}$ and by definition of $\mu_{n,k}$, we clearly have:
\[
 d_\text{H}^{d^{X_n}}(X_n^{(\ell_k)},S^{n,+}_{\ell_k,k}) \le 2^{-k-1} 
\quad\text{and}\quad
 d_\text{P}^{d^{X_n}} (\mu_n^{(\ell_k)}, \mu_{n,k}
 \ind_{S^{n,+}_{\ell_k,k}}) \le 2^{-k}.
\]
By considering the identity map from $S^{n,+}_{\ell_k,k}$ to
$X^{(\ell_k)}$, we deduce that:
\begin{equation}
 \label{eq:B2}
B_n^2= d_{\text{GHP}}^c\left(\cx_n^{(\ell_k)}, \cy^n_k \right)\leq
2^{-k+1}.
\end{equation}

Recall the correspondence \reff{eq:correspondence}. It is easy to check
that the
function defined on $\left(S_{\ell_k,k}^{n,+}\sqcup S_{\ell_k,k}^{+}
\right)^2$ by:
\begin{equation}
 d_n(y,z) = 
\begin{cases}
 d^{X_n}(y,z) & \text{if}\ y,z\in S_{\ell_k,k}^{n,+},\\
	 d(y,z) & \text{if}\ y,z\in S_{\ell_k,k}^{+},\\
	 \inf\{ d^{X_n}(y,y')+d(z,z')
+\frac{1}{2} \delta_n( \ell_k,k)
; \ (y',z')\in \crr^{n,+}_{\ell_k,k} \} & \text{if}\ y\in S_{\ell_k,k}^{n,+}
,z\in S_{\ell_k,k}^{+}
 \end{cases}
\end{equation}
is a metric. For this particular metric, we easily have
$d_n(\emptyset_n,\emptyset)\leq \inv{2} \delta_n( \ell_k,k)$ as well as:
\[
d_\text{H}^{d_n}(S_{\ell_k,k}^{n,+},S_{\ell_k,k}^{+}) \leq \frac{1}{2}
\delta_n( \ell_k,k)
\quad\text{and}\quad
d_\text{P}^{d_n}(\mu_{n,k}^{(\ell_k)} ,\nu_{n,k}^{\{\ell_k\}}
) \leq \frac{1}{2}
\delta_n( \ell_k,k). 
\]
We deduce that:
\begin{equation}
 \label{eq:B3}
B_n^3= d_{\text{GHP}}^c\left(\cy^n_k, \cz_k^n \right)\leq
\frac{3}{2}\delta_n( \ell_k,k). 
\end{equation}

As $S^{+}_{\ell_k,k} $ is a $2^{-k}$-net of
$X^{\ell_k}$, thanks to Lemma \ref{lem:approx}, we get:
\begin{equation}
 \label{eq:B4}
B_n^4
= d_{\text{GHP}}^c\left(\cz_k^n ,\cw^n_k\right)\leq 2^{-k}.
\end{equation}

Concerning $B_n^5$, we only need to bound the Prokhorov distance between
$ \nu_{n,k}^{\{\ell_k\}}$ and $\nu_{n,k}^{(\ell_k)}$. Recall that
$\nu_{n,k}^{\{\ell_k\}}\leq \nu_{n,k}^{(\ell_k)} $ and that
$\nu_{n,k}^{(\ell_k)} $ may differ only on $\partial_{\ell_k}
X$. For $A$ closed, we have:
\[
\nu_{n,k}^{\{\ell_k\}}(A)\leq
\nu_{n,k}^{(\ell_k)}(A)
\quad\text{and}\quad
\nu_{n,k}^{(\ell_k)}(A)\leq \nu_{n,k}^{\{\ell_k\}}(A) +
\nu_{n,k}(\partial_{\ell_k} X).
\]
Recall \reff{eq:A-nu}. Let $\rho(r)\geq r+3$ such that $\rho(r)\not\in
A_\nu$ and:
\begin{equation}
 \label{eq:def-em}
\varepsilon_{n,k}=2d_\text{P}(\nu_{n,k}^{(\rho(r))}, \nu^{(\rho(r))}).
\end{equation}
As $\ell_k\leq r+2^{-k}$, we have:
\[
\nu_{n,k}(\partial_{\ell_k} X)\leq \nu ((\partial_{\ell_k}
X)^{\varepsilon_{n,k}})+ \varepsilon_{n,k}
\leq \nu (X^{(r+2^{-k}+\varepsilon_{n,k})}\backslash X^{(r
 -2\varepsilon_{n,k})}) 
+\varepsilon_{n,k}. 
\]
We deduce that:
\begin{equation}
 \label{eq:B5}
B_n^5
= d_{\text{GHP}}^c \left( \cw^n_k, \cx^{(\ell_k)}\right) \leq \nu
(X^{(r+2^{-k}+\varepsilon_{n,k})}\backslash X^{(r 
 -2\varepsilon_{n,k})}) +\varepsilon_{n,k}. 
\end{equation}

Lemma \ref{df:estimation}
and the fact that $X$ is a length space gives:
\begin{equation}
 \label{eq:B6}
B_n^6
=d_{\text{GHP}}^c \left( \cx^{(\ell_k)}, \cx^{(r)}\right)
\leq 2^{-k}+ \nu (X^{(\ell_k)}\backslash X^{(r)}). 
\end{equation}

Putting \reff{eq:B1}, \reff{eq:B2}, \reff{eq:B3}, \reff{eq:B4},
\reff{eq:B5}, \reff{eq:B6} in \reff{eq:majo-dghpc}, we get:
\begin{multline}
 \label{eq:majo-dghpc2}
 d_{\text{GHP}}^c(\cx_n^{(r)},\cx^{(r)})
\leq 5\cdot 2^{-k}+ \mu_n(X_n^{(\ell_k)}\backslash X_n^{(r)} )
\\+ \frac{3}{2}\delta_n( \ell_k,k)
+ \nu
(X^{(r+2^{-k}+\varepsilon_{n,k})}\backslash X^{(r 
 -2\varepsilon_{n,k})}) +\varepsilon_{n,k} + \nu (X^{(\ell_k)}\ X^{(r)}). 
\end{multline}

We give a more precise upper bound for $\mu_n(X_n^{(\ell_k)}\backslash
X_n^{(r)} )$. Using arguments similar to those used to get
\reff{eq:majo-mu-nu}, we have:
\begin{align*}
\mu_n(X_n^{(\ell_k)}\backslash X_n^{(r)} )
&\leq \mu_n(X_n^{(\ell_k)}) -\mu_n(X_n^{(\ell_k-2^{-k})})\\
&\leq \nu_{n,k}( X^{(\ell_k+ \delta_n(\ell_k,k)+2^{-k})}) - 
\nu_{n,k}( X^{(\ell_k- \delta_n(\ell_k,k)-4\cdot 2^{-k})}) . 
\end{align*}
For $k\geq r+1$, we have $\delta_n(\ell_k,k)\leq \delta_n(k,k)$ thanks
to \reff{eq:ineg-delta}. Then using the sub-sequence $(n_k, k\in \N)$
defined at the end of Section \ref{sec:approx-meas} with \reff{eq:eta},
we get that:
\begin{align*}
 \mu_{n_k}(X_ {n_k}^{(\ell_k)}\backslash X_{n_k}^{(r)} )
&\leq \nu_{n_k,k}( X^{(\ell_k+ 2\cdot 2^{-k})}) - 
\nu_{n_k,k}( X^{(\ell_k -5\cdot 2^{-k})}) \\
&\leq \nu_{}( X^{(\ell_k+ 2\cdot 2^{-k}+\varepsilon_{n_k,k})}) - 
\nu( X^{(\ell_k-5\cdot 2^{-k}- \varepsilon_{n_k,k})}) + 2\varepsilon_{n_k,k}.
\end{align*}
Notice that the sub-sequence $(n_k, k\in\N)$ does not depend on $r$: it is the
same for all $r\geq 0$. Using \reff{eq:majo-dghpc2}, we get for $k\geq r+1$:
\[
 d_{\text{GHP}}^c(\cx_{n_k}^{(r)},\cx^{(r)})
\leq 5\cdot 2^{-k}+ \frac{3}{2}\eta_k 
+ 2\nu
(X^{(\ell_k+2^{-k}+\varepsilon_{n,k})}\backslash X^{(\ell_k -5\cdot 2^{-k} 
 -2\varepsilon_{n,k})}) +3\varepsilon_{n_k,k}.
\]
As $\lim_{k\rightarrow+\infty } \ell_k=r$ and $\lim_{k\rightarrow+\infty
} \varepsilon_{n_k,k}=0$, we get using \reff{eq:eta}, that for $r\not
\in A_\nu$:
\[
\lim_{k\rightarrow+\infty }
d_{\text{GHP}}^c(\cx_{n_k}^{(r)},\cx^{(r)})=0.
\]
By dominated convergence, we get that $\lim_{k\rightarrow+\infty }
d_{\text{GHP}}(\cx_{n_k},\cx) =0$. Thus we have a converging sub-sequence
in $\cc$.

\subsection{Proof of (ii) of Theorem \ref{theo:LL}}

We need to prove that the metric space $(\LL, d_\text{GHP})$ is
separable and complete.

\begin{lem}
 \label{lem:sep}
The metric space $(\LL, d_\text{GHP})$ is separable.
\end{lem}

\begin{proof}
 We can notice that the set $\K\cap\LL$ is dense in
 $(\LL,d_{\text{GHP}})$, since for $\cx\in \LL$, for all $r>0$ we have
 $\cx^{(r)}\in \K$ and $d_{\text{GHP}}(\cx^{(r)},\cx)\le
 \expp{-r}$. Every element of $\K$ can be approximated in the
 $d_{\text{GHP}}^c$ topology by a sequence of metric spaces with finite
 cardinal, rational edge-lengths and rational weights. Hence,
 $(\K\cap\LL,d_{\text{GHP}}^c)$ is separable, being a subspace of a
 separable metric space. According to Proposition \ref{prop:K-LL},
 $(\K\cap\LL,d_{\text{GHP}})$ is also separable. As $\K\cap\LL$ is
 dense in $(\LL,d_{\text{GHP}})$, we deduce that $(\LL,d_{\text{GHP}})$
 is separable. 
\end{proof}

\begin{lem}
 \label{lem:complet}
The metric space $(\LL, d_\text{GHP})$ is complete.
\end{lem}

\begin{proof}
 Let $(\cx_n, n\in \N)$, with $\cx_n=(X_n, d^{X_n}, \emptyset_n,
 \mu_n)$, be a Cauchy sequence in $(\LL, d_\text{GHP})$. It is enough to
 prove that it is relatively compact. Thus, we need to prove it satisfies
 condition (i) and (ii) of Theorem \ref{GHP:PreComL}.\\

Assume there exists $r_0\in \R_+$ such that $\sup_{n\in \N}
\mu_n(X_n^{(r_0)})=+\infty $. By considering a sub-sequence, we may
assume that $\lim_{n\rightarrow+\infty } \mu_n(X_n^{(r_0)})=+\infty
$. This implies that for any $r\geq r_0$, $\lim_{n\rightarrow+\infty }
\mu_n(X_n^{(r)})=+\infty$. Thus, we have for any $m\in \N$:
\[
\lim_{n\rightarrow+\infty } \int_0^{+\infty } \expp{-r}\; \left(1\wedge
 \val{ \mu_n(X_n^{(r)}) - \mu_m(X_m^{(r)})} 
\right)\; dr \geq \expp{-r_0}.
\]
Then use \reff{eq:diamP} to get that $(\cx_n, n\in \N)$ is not a
Cauchy sequence. Thus, if $(\cx_n, n\in \N)$ is a
Cauchy sequence, then (ii) of Theorem \ref{GHP:PreComL} is satisfied. \\

Let $g_{n,m}(r)=d_\text{GH}^c((X_n^{(r)},d^{X_n^{(r)}}),
(X_m^{(r)},d^{X_m^{(r)}})) $. 
On the one hand, use \reff{eq:GHP-GH} to get:
\begin{equation}
 \label{eq:cv-gnm}
\lim_{\min(n,m)\rightarrow+\infty }
\int_0^{+\infty } \expp{-r}\; \left(1\wedge
 g_{n,m}(r)
\right)\; dr =0.
\end{equation}
On the other hand, using \reff{eq:GHP-GH} and Lemma
\ref{df:estimation}, and arguing as in the proof of Lemma
\ref{lem:reg-d-GHP}, we get that for any $r,\varepsilon\geq 0$:
\[
|g_{n,m}(r)- g_{n,m}(r+\varepsilon)|\leq 2 \varepsilon.
\]
This implies the functions $g_{n,m}$ are 2-Lipschitz. Thus, we deduce
from \reff{eq:cv-gnm}, that for all $r\geq 0$,
$\lim_{\min(n,m)\rightarrow+\infty } g_{n,m}(r)=0$. Thus the sequence
$((X_n^{(r)}, d^{X_n^{(r)}}), n\in \N)$ is a Cauchy sequence for the
Gromov-Hausdorff metric. Then point (2) of Proposition 7.4.11 in
\cite{Burago2001} readily implies condition (i) of Theorem
\ref{GHP:PreComL}.
\end{proof}

\end{document}